\documentclass[a4paper,12pt]{amsart}

\usepackage{amsmath, amssymb, amsthm}

\newcommand{\bP}{\mathbb{P}}
\newcommand{\bN}{\mathbb{N}}
\newcommand{\bC}{\mathbb{C}}
\newcommand{\bR}{\mathbb{R}}

\newcommand{\rd}{\mathrm{d}}% d for measures in integration

\newcommand{\supp}{\operatorname{supp}}%support of measures etc
\theoremstyle{plain}
\newtheorem{theorem}[equation]{Theorem}
\newtheorem{lemma}[equation]{Lemma}

\newtheorem{claim}{Claim}
\newtheorem{mainth}{Theorem}
\theoremstyle{definition}

\newtheorem*{ac}{Acknowledgment}
\theoremstyle{remark}

\numberwithin{equation}{section}

\begin{document}
\title[The converse of Brolin's theorem]{A generalization of the converse of Brolin's theorem}

\author[Y\^usuke Okuyama]{Y\^usuke Okuyama}
\address{
Division of Mathematics,
Kyoto Institute of Technology, Sakyo-ku,
Kyoto 606-8585 Japan}
\email{okuyama@kit.ac.jp}
\author[Ma{\l}gorzata Stawiska]{Ma{\l}gorzata Stawiska}
\address{Mathematical Reviews, 416 Fourth St., Ann Arbor, MI 48103, USA}
\email{stawiska@umich.edu}

\date{\today}

\subjclass[2010]{Primary 37F10; Secondary 31A15}
\keywords{equilibrium measure, harmonic measure,
Brolin's theorem, Lopes's theorem, complex dynamics, potential theory}

\begin{abstract}
We prove a 
generalization of Lopes's theorem,
that is, of the converse of Brolin's theorem.
\end{abstract}

\maketitle

\section{Introduction}\label{sec:intro}

Let $f$ be a rational function on $\bP^1=\bC\cup\{\infty\}$ of degree $d>1$.
Let us denote by $J(f)$ the Julia set of $f$, and by $F(f)$ the Fatou set 
$\bP^1\setminus J(f)$ of $f$. 
Let $\omega$ be the Fubini-Study area element on $\bP^1$
normalized so that $\omega(\bP^1)=1$.
Then the weak limit 
\begin{gather*}
 \mu_f:=\lim_{n\to\infty}\frac{(f^n)^*\omega}{d^n}
\end{gather*}
exists on $\bP^1$, has no atoms in $\bP^1$, and charges no
polar subsets in $\bP^1$. This probability measure $\mu_f$ is called 
{\itshape the equilibrium $($or the $($non-exceptionally$)$ balanced$)$ measure} 
of $f$ on $\bP^1$, and
is in fact the unique
probability measure $\nu$ on $\bP^1$ such that 
$f^*\nu=d\cdot\nu$ on $\bP^1$ and that $\nu(E(f))=0$,
where the exceptional set
$E(f):=\{a\in\bP^1:f^{-2}(a)=\{a\}\}$ of $f$ consists of at most two points 
in $\bP^1$ (for any $a\in E(f)$, the probability 
measure $\nu_a:=(\delta_a+\delta_{f(a)})/2$ on $\bP^1$
also satisfies $f^*\nu_a=d\cdot\nu_a$ on $\bP^1$, that is, $\nu_a$ is balanced under $f$, but $\nu_a(E(f))>0$). In particular,  $J(f)=\supp\mu_f$
and $J(f)$ is non-polar.

When $\infty\in F(f)$, let us denote by $D_\infty=D_\infty(f)$ 
the Fatou component of $f$ containing $\infty$, and by 
\begin{gather*}
 \nu_\infty=\nu_{D_\infty,\infty} 
\end{gather*}
the {\itshape harmonic measure} of $D_\infty$
with pole $\infty$, which is a probability measure on $\partial D_\infty$. 
This measure $\nu_\infty$ exists since $J(f)$ is non-polar as mentioned
in the above.

Our aim in this short notes is to prove the following.

\begin{mainth}\label{th:polynomial}
 Let $f$ be a rational function on $\bP^1$ of degree $d>1$, 
 and suppose that $\infty\in F(f)$.
 Then the following are equivalent$;$
\begin{enumerate}
 \item $f^2$ is a polynomial. \label{item:poly}
 \item $\mu_f=\nu_\infty$ on $\bP^1$.  \label{item:equal}
\end{enumerate}
\end{mainth}

The implication (\ref{item:poly})$\Rightarrow$(\ref{item:equal}) in Theorem \ref{th:polynomial} follows from  the work of Brolin \cite{Brolin}, 
so we will show the converse implication
(\ref{item:equal})$\Rightarrow$(\ref{item:poly}). 
The statement (\ref{item:poly}) is equivalent to the following statement that  
{\itshape either $f$ is a polynomial 
or $f^{-2}(\infty)=\{\infty\}\not\subset f^{-1}(\infty)$},
the latter possibility in which
never occurs if $f(D_\infty)=D_\infty$ and is in fact equivalent to 
{\itshape $f$ having the form $a(z-b)^{-d}+b$ 
for some $a \in \mathbb{C}^*$ and some $b \in \mathbb{C}$}. 
In particular, the iteration order $2$ of $f$ in the statement (\ref{item:poly})
is best possible.

The implication (\ref{item:equal})$\Rightarrow$(\ref{item:poly}) 
in Theorem \ref{th:polynomial} was first claimed
by Oba and Pitcher \cite{ObaPitcher}, also assuming
$f(\infty)=\infty$ (so $f(D_\infty)=D_\infty$) and $f'(\infty)=0$. 
It was established by Lopes \cite{Lopes86} (see also Lalley \cite{Lalley92}
and Ma\~n\'e--da Rocha \cite{ManeDaRocha92}) under a 
relaxed additional assumption $f(\infty)=\infty$, 
and by the present authors \cite[Theorem 1]{OS11} under
$f(D_\infty)=D_\infty$.

Theorem \ref{th:polynomial} is not much stronger than 
\cite[Theorem 1]{OS11}. Indeed, with a little extra effort,
it can be obtained by combining \cite[Theorem 1]{OS11}
with Sullivan's no wandering domain theorem \cite{Sullivan85}
and the Riemann-Hurwitz formula. 
In what follows, we will prove Theorem \ref{th:polynomial} in an alternative way. 
We will first improve some part (see Claim \ref{th:totinv} below) in the proof of 
the implication \cite[(\ref{item:equal})$\Rightarrow$(\ref{item:poly}) in Theorem 1]{OS11} and then give a simple proof of the remaining part,
using the following pleasant theorem due to Orevkov.

\begin{theorem}[Orevkov {\cite[{a consequence of Corollary 1}]{orevkov18}}]\label{th:invcircle}
For every $P\in\bC[z]$ of degree $>0$,
the lemniscate $\{z\in\bC:|P(z)|=1\}$ is an irreducible real-algebraic curve
in $\bC\cong\bR^2$ 
$($identifying $z\in\bC$ with $(\Re z,\Im z)\in\bR^2)$, 
that is, this lemniscate coincides with 
the zero set $\{(x,y)\in\bR^2:\phi(x,y)=0\}$ of some $\phi(x,y)\in\bR[x,y]$
irreducible as an element of $\bC[x,y]$.
\end{theorem}

We conclude this section with some background material.
For more details, see e.g.\ the books \cite{Ransford95,Milnor3rd}, respectively.

\subsection*{Potential theory}
For every probability measure $\nu$ on $\bC$ having the compact support,
let $p_\nu$ be the {\itshape $($logarithmic$)$ potential} of $\nu$ on $\bC$
with pole $\infty$
so that $p_\nu(z)=\int_\bC\log|z-w|\nu(w)$ on $\bC$. 
Then $\rd\rd^c p_\nu=\nu-\delta_\infty$ on $\bP^1$ 
(as a $\delta$-subharmonic function on $\bP^1$) and 
$p_\nu(z)=\log|z|+O(|z|^{-1})$ as $z\to\infty$.
We also set 
\begin{gather*}
 I_\nu:=\int_{\bC}p_\nu\nu>-\infty, 
\end{gather*}
which is called the energy of $\nu$ with pole $\infty$.

Let $D$ be a domain in $\bP^1$
containing $\infty$ such that $\bC\setminus D$ is non-polar (i.e.,
$I_\nu>-\infty$ for some probability measure $\nu$ supported by $\bC\setminus D$).
The {\itshape equilibrium mass distribution} on $\bC\setminus D$
with pole $\infty$ is the unique
probability measure $\nu$ on $\bC\setminus D$ such that $p_\nu\ge I_\nu$ on $\bC$,
$p_\nu>I_\nu$ on $D$, and $p_\nu\equiv I_{\nu}$
on $\bC\setminus D$ except for some (possibly empty)
$F^\sigma$-polar subset in $\partial D$. 
This probability measure is supported by $\partial D$ and 
coincides with the {\itshape harmonic measure} $\nu_{D,\infty}$
of $D$ with pole $\infty$.

\subsection*{Complex dynamics}
Let $f$ be a rational function on $\bP^1$ of degree $d>1$.
The Julia set $J(f)$ of $f$ is defined by the set of all non-normality points
of the family $(f^n)_{n\in\bN}$, and the Fatou set $F(f)$ of $f$ by $\bP^1\setminus J(f)$.

A component of $F(f)$ is called a Fatou component of $f$. A Fatou component
of $f$ is properly mapped by $f$ to a Fatou component of $f$, and
the preimage of a Fatou component of $f$ under $f$ consists of
(at most $d$) Fatou components of $f$. 

Let $F(z_0,z_1)=(F_0(z_0,z_1),F_1(z_0,z_1))\in(\bC[z_0,z_1]_d)^2$
be an ordered pair of homogeneous polynomials $F_0,F_1$ of degree $d$
such that   $f(z)=F_1(1,z)/F_0(1,z)$. Such an $F$ 
is unique up to multiplication in $\bC^*$  and is called
a (non-degenerate homogeneous) {\itshape lift} of $f$.
We note that for every $n\in\bN$, $\deg(f^n)=d^n$, and
the $n$-th iterate {$F^n$} of $F${, which is written as}
\begin{gather*}
 F^n(z_0,z_1)=\bigl(F_0^{(n)}(z_0,z_1),F_1^{(n)}(z_0,z_1)\bigr)\in(\bC[z_0,z_1]_{d^n})^2,
\end{gather*}
is a lift of $f^n$.
The {\itshape uniform} limit
\begin{gather*}
 G^F:=\lim_{n\to\infty}\frac{\log\|F^n\|}{d^n}\quad\text{on }\bC^2\setminus\{(0,0)\} 
\end{gather*}
exists,
where $\|\cdot\|$ is the Euclidean norm on $\bC^2$, and is called 
the {\itshape  escape rate} function of $F$ on $\bC^2\setminus\{(0,0)\}$.
In particular, the function $G^F$ 
is a continuous and plurisubharmonic function on $\bC^2\setminus\{(0,0)\}$,
and we have the equality
\begin{gather*}
 G^F(cZ)=G^F(Z)+\log|c| 
\end{gather*}
for every $c\in\bC^*$ and every $Z\in\bC^2\setminus\{(0,0)\}$, the equality
\begin{gather*}
 G^F\circ F=d\cdot G^F
\end{gather*}
on $\bC^2\setminus\{(0,0)\}$, the identity
\begin{gather*}
 G^{cF}=G^F+\frac{\log|c|}{d-1}\quad\text{on }\bC^2\setminus\{(0,0)\} 
\end{gather*}
for every $c\in\bC^*$,
and the equality $\rd\rd^c G^F(1,\cdot)=\mu_f-\delta_\infty$ on $\bP^1$
(as a $\delta$-subharmonic function on $\bP^1$).
We also note that for every $n\in\bN$,
$G^{F^n}=G^F$ on $\bC^2\setminus\{(0,0)\}$ and 
$\mu_{f^n}=\mu_f$ on $\bP^1$.
%\begin{remark}
% The function $G^F$ is known to be H\"older continuous, but we will not use this stronger property.
%\end{remark}

\section{Proof of Theorem \ref{th:polynomial}}
Let $f$ be a rational function on $\bP^1$ of degree $d>1$.
As we already mentioned in Section \ref{sec:intro}, we only need to show 
the implication (\ref{item:equal})$\Rightarrow$(\ref{item:poly}). 

Fix a lift $F$ of $f$, and suppose that $\infty\in F(f)$. 

\subsection*{Preliminary lemmas} Let us recall the following 
from \cite[\S 3]{OS11}. 

\begin{lemma}\label{th:dynamical}
The potential $p_{\mu_f}$ is continuous on $\bC$. More precisely,
\begin{gather}
  G^F(1,\cdot)=p_{\mu_f}+G^F(0,1)\quad\text{on }\bC.\label{eq:dynamical}
\end{gather}
\end{lemma}

\begin{proof}
 The values on $\bP^1$ of the $\rd\rd^c$ operator on
 both $G^F(1,\cdot)$ and $p_{\mu_f}$ 
 (as $\delta$-subharmonic functions on $\bP^1$) are
 $\mu_f-\delta_\infty$. Hence we have
 $G^F(1,\cdot)-p_{\mu_f}\equiv C$ on $\bP^1$ for some $C\in\bR$. 
 Moreover, we have
 $C=\lim_{z\to\infty}(G^F(1/z,1)-(p_{\mu_f}-\log|z|))=G^F(0,1)$.
\end{proof}

\begin{lemma}[the pullback formula of $p_{\mu_f}$ under $f$]\label{th:pullback}
For every $n\in\bN$,
\begin{multline}
 p_{\mu_f}\circ f^n+\log\bigl|F_0^{(n)}(1,\cdot)\bigr|\\
=d^n\cdot p_{\mu_f}+(d^n-1)G^F(0,1)\quad\text{on }
\bC\setminus f^{-n}(\infty).\label{eq:pullback} 
\end{multline}
\end{lemma}

\begin{proof}
Without loss of generality, we can assume that $n=1$.
Using Lemma \ref{th:dynamical}, we can compute as
\begin{align*}
 p_{\mu_f}\circ f=&G^F(1,f(\cdot))-G^F(0,1)\\
=&G^F\circ F-\log|F_0(1,\cdot)|-G^F(0,1)\\
=&d\cdot G^F-\log|F_0(1,\cdot)|-G^F(0,1)\\
=&d(p_{\mu_f}+G^F(0,1))-\log|F_0(1,\cdot)|-G^F(0,1)\\
=&d\cdot p_{\mu_f}-\log|F_0(1,\cdot)|+(d-1)G^F(0,1)
\end{align*} 
on $\bC\setminus f^{-1}(\infty)$.
\end{proof}

\subsection*{Proof of (\ref{item:equal})$\Rightarrow$(\ref{item:poly})}
Suppose now that $\mu_f=\nu_\infty$ on $\bP^1$ 
(then $p_{\mu_f}\equiv I_{\mu_f}$ on $\bC\setminus D_\infty$),
and suppose that $f$ is not a polynomial,
that is, $F_0(1,\cdot)$ is non-constant on $\bC$.

Then by \eqref{eq:pullback} for $n=1$ (and $p_{\mu_f}\equiv I_{\mu_f}$ on $\bC\setminus D_\infty$), we have 
\begin{gather}
 |F_0(1,\cdot)|\equiv e^{(d-1)(I_{\mu_f}+G^F(0,1))}
\quad\text{on }\bC\setminus(D_\infty\cup f^{-1}(D_\infty)).
 \label{eq:lemniscate}
\end{gather}

%Let us see the following, without assuming $f(D_\infty)=D_\infty$.

\begin{claim}\label{th:totinv}
 Either $D_\infty=F(f)$, 
 or $f^2$ is a polynomial.
\end{claim}

\begin{proof}
 It is clear that $F(f)\supset D_\infty \cup f^{-1}(D_\infty)$.
 We claim that $F(f)=D_\infty \cup f^{-1}(D_\infty)$; 
for, otherwise, there is a Fatou component
 $U$ of $f$ in $F(f) \setminus (D_\infty \cup f^{-1}(D_\infty))$. 
 This is impossible by \eqref{eq:lemniscate}
 (and since $F_0(1,\cdot)$ is non-constant on $\bC$).

 Hence $F(f)=D_\infty \cup f^{-1}(D_\infty)$. 
 Then by \eqref{eq:lemniscate} (and since $F_0(1,\cdot)$ is non-constant on $\bC$),
 the Fatou component $f(D_\infty)$ is either $D_\infty$ or 
 a component of $f^{-1}(D_\infty)(\setminus D_\infty)$.
In the latter case we in fact have $f^2(D_\infty)=D_\infty$.
 
 Suppose first that $f(D_\infty)=D_\infty$. 
 We claim that $f^{-1}(D_\infty)=D_\infty$;
 for, otherwise, there is a Fatou component $U$ of $f$ 
 in $f^{-1}(D_\infty)\setminus D_\infty$, and then
 $(\emptyset\neq)f^{-1}(U)\subset F(f)\setminus(D_\infty \cup f^{-1}(D_\infty))$. 
 This is impossible by \eqref{eq:lemniscate}
 (and since $F_0(1,\cdot)$ is non-constant on $\bC$). 
Once $f^{-1}(D_\infty)=D_\infty$ is at our disposal, we have $F(f)=D_\infty$
 by \eqref{eq:lemniscate} (and since $F_0(1,\cdot)$ is non-constant on $\bC$).
 
 Suppose next that $f^2(D_\infty)=D_\infty$. 
 We note that by \eqref{eq:pullback} for $n=2$ 
 (and $p_{\mu_f}\equiv I_{\mu_f}$ on $\bC\setminus D_\infty$), 
 we also have 
\begin{gather}
 \bigl|F_0^{(2)}(1,\cdot)\bigr|\equiv e^{(d^2-1)(I_{\mu_f}+G^F(0,1))}
\quad\text{on }\bC\setminus(D_\infty\cup f^{-2}(D_\infty)).\tag{\ref{eq:lemniscate}$'$}\label{eq:lemniscate_2}
\end{gather} 
We claim that either $F(f)=D_\infty$, or $f^2$ is a polynomial; 
 indeed,  (a) if $f^{-2}(D_\infty)\neq D_\infty$, then 
 there is a Fatou component $U$ of $f$ 
 in $f^{-2}(D_\infty)\setminus D_\infty$, and then
 $(\emptyset\neq)f^{-2}(U)\subset F(f)\setminus(D_\infty \cup f^{-2}(D_\infty))$.
 Hence by \eqref{eq:lemniscate_2}, $F_0^{(2)}(1,\cdot)$ is constant on $\bC$, 
 that is, $f^2$ is a polynomial. 
 (b) If $f^{-2}(D_\infty)=D_\infty$, then 
 by \eqref{eq:lemniscate_2}, we have 
 $|F_0^{(2)}(1,\cdot)|\equiv e^{(d^2-1)(I_{\mu_f}+G^F(0,1))}$ on $\bC\setminus D_\infty$, so that either $F(f)=D_\infty$, or
 $F_0^{(2)}(1,\cdot)$ is constant on $\bC$. The latter possibility is 
 equivalent to $f^2$ being a polynomial.
\end{proof}

Choose $c\in\bC^*$ satisfying $|c|=e^{-(d-1)(I_{\mu_f}+G^F(0,1))}$, so that
\begin{gather}
 G^{cF}(0,1)=-I_{\mu_f}.\label{eq:normalized}
\end{gather}
Then for every $n\in\bN$, by \eqref{eq:pullback} (applied to
the lift $cF$ of $f$) and $p_{\mu_f}\equiv I_{\mu_f}$ on $\bC\setminus D_\infty$, 
the lemniscate 
\begin{gather*}
  L_{(cF)^n}:=\bigl\{z\in\bC:\bigl|(cF)_0^{(n)}(1,z)\bigr|=1\bigr\}
\end{gather*} 
contains 
$J(f)(\subset\bC\setminus f^{-n}(D_\infty)\subset\bC\setminus f^{-n}(\infty)$).

\begin{claim}\label{th:coincidence}
For every $n\in\bN$, $L_{(cF)^n}=L_{cF}$. Moreover, $f(L_{cF})\subset L_{cF}$. 
\end{claim}

\begin{proof}
 For every $n\in\bN$, the non-polar (so infinite) set $J(f)$
 is contained in both $L_{cF}$ and $L_{(cF)^n}$. Hence 
(identifying $z\in\bC$ with $(\Re z,\Im z)\in\bR^2$,) we have
\begin{gather*}
  L_{cF}=\{(x,y)\in\bR^2:\phi(x,y)=0\}=L_{(cF)^n}
\end{gather*} 
for some $\phi(x,y)\in\bR[x,y]$ irreducible as an element of $\bC[x,y]$,
by Orevkov's theorem (Theorem \ref{th:invcircle}) and
the B\'ezout theorem (see, e.g., the book \cite[\S I.7]{Hartshorne77};
{for our purpose, a more elementary
\cite[Page 4, Lemma]{Shafarevich3rd}
is enough}). 
Moreover, for every $z\in\bC$, 
we have
$(cF)_0^{(2)}(1,z)=(cF_0)((cF)_0(1,z),(cF)_1(1,z))
 =(cF)_0(1,f(z))\cdot((cF)_0(1,z))^d$. 
Hence for every $z\in L_{cF}(=L_{(cF)^2}$), 
 we have $1=|(cF)_0(1,f(z))|\cdot 1^d$, that is,
 $f(L_{cF})\subset L_{cF}$.
\end{proof}

Suppose finally that $D_\infty=F(f)$. Then $L_{cF}\cap F(f)\neq\emptyset$;
for, otherwise, we must have $L_{cF}=J(f)$, so that $F(f)$ has also a
bounded component (intersecting with $f^{-1}(\infty)$). This contradicts $D_\infty=F(f)$.

Pick $z_0\in L_{cF}\cap F(f)$ (then $z_0\in L_{(cF)^n}$ for every $n\in\bN$
by the equality $L_{cF}=L_{(cF)^n}$ in Claim \ref{th:coincidence}, 
and $p_{\mu_f}(z_0)>I_{\mu_f}$).
Then by \eqref{eq:pullback} (applied to the lift $cF$ of $f$)
and \eqref{eq:normalized}, we must have
$p_{\mu_f}(f^n(z_0))-I_{\mu_f}
=d^n(p_{\mu_f}(z_0)-I_{\mu_f})\to\infty$ as $n\to\infty$.

On the other hand, by the inclusion $f(L_{cF})\subset L_{cF}$
in Claim \ref{th:coincidence}, 
the boundedness of $L_{cF}$ in $\bC$,
and the (upper semi)continuity of $p_{\mu_f}$ on $\bC$, 
we have 
$\limsup_{n\to\infty}(p_{\mu_f}(f^n(z_0))-I_{\mu_f})
\le\sup_{L_{cF}}(p_{\mu_f}-I_{\mu_f})<\infty$.
This is a contradiction. \qed

\begin{ac}
 The first author was partially supported by JSPS Grant-in-Aid 
 for Scientific Research (C), 15K04924.
\end{ac}

\def\cprime{$'$}

\end{document}